\documentclass[dvipsnames]{amsart}

\usepackage[T1]{fontenc}
\usepackage[utf8]{inputenc}

\usepackage{amsmath}
\usepackage{amssymb}
\usepackage{amsthm}
\usepackage{physics}  
\usepackage{upgreek}
\usepackage{mathtools}    
\usepackage[final]{microtype}

\usepackage{tikz}
\usepackage{tikz-cd}
\usetikzlibrary{arrows, babel, 
  decorations.text, patterns, 
  patterns.meta, decorations.pathreplacing}
\tikzset{commutative diagrams/.cd,
  arrow style=tikz, diagrams={>=stealth}}

\usepackage{url}
\usepackage[
            pdfnewwindow=false,
            colorlinks=true,
            allcolors=black,
            urlcolor=red,
            citecolor=green!50!black]{hyperref}
\usepackage[safeinputenc, 
            backend=biber, 
            url=false,
            doi=false,
            isbn=false,
            bibencoding=auto,
            style=alphabetic, 
            giveninits=true,
            backref=true]{biblatex}

\DeclareFieldFormat[article, incollection]{title}{\mkbibemph{#1}} 

\AtEveryBibitem{%
   \ifentrytype{misc}{%
      \clearfield{year}%
   }{}%
}

\DefineBibliographyStrings{english}{%
  backrefpage = {$\uparrow$},
  backrefpages = {$\uparrow$},}

\AtBeginBibliography{\small}
\renewbibmacro{in:}{}

\usepackage{CJKutf8}
\begin{CJK*}{UTF8}{goth}
\gdef\yama{\mbox{\textbf{山}}}
\gdef\ten{\mbox{\textbf{天}}}
\end{CJK*}

\usepackage{longtable}
\usepackage{booktabs}
\usepackage{caption}
\usepackage{multirow}

\usepackage{enumitem}
\setenumerate[1]{label={{\bf \arabic*}.}, itemsep=3pt}
\setenumerate[2]{label=(\alph*)}
\numberwithin{equation}{section}
\numberwithin{table}{section}

\usepackage{fancyhdr}

\usepackage[includehead, includefoot, 
  left=3cm, right=3cm, top=2cm, bottom=3cm, 
  footskip=5mm]{geometry}

\theoremstyle{plain}
\newtheorem{theorem}{Theorem}
  \newtheorem{proposition}[theorem]{Proposition}
  \newtheorem{lemma}[theorem]{Lemma}
  
  \newtheorem{conjecture}[theorem]{Conjecture}

\newtheorem*{theorem*}{Theorem}

\theoremstyle{remark}
\newtheorem{remark}[theorem]{Remark}
\newtheorem*{remark*}{Remark}

\newcommand{\FF}{\mathbb{F}}

\newcommand{\PP}{\mathbb{P}}

\newcommand{\QQ}{\mathbb{Q}}

\newcommand{\Qbar}{\overline{\mathbb{Q}}}

\newcommand{\pp}{\mathfrak{p}}
\newcommand{\bmu}{\mathbf{\upmu}}

\DeclareMathOperator{\Gal}{\textnormal{\textsf{Gal}}}
\DeclareMathOperator{\GL}{\textnormal{\textsf{GL}}}

\title{Equivalence of Conjectures on Heavenly Elliptic Curves}

\author[C.~McLeman]{Cam McLeman}
\address{Mathematics Department \\ University of Michigan-Flint \\ Flint MI, 48502 \\ United States}
\email{mclemanc@umich.edu}

\author[C.~Rasmussen]{Christopher Rasmussen}
\address{Department of Mathematics and Computer Science \\ Wesleyan University \\ Middletown CT, 06459 \\ United States}
\email{crasmussen@wesleyan.edu}

\begin{document}

\begin{abstract}
  Heavenly abelian varieties are abelian varieties defined over number fields that exhibit constrained $\ell$-adic Galois representations for some rational prime $\ell$. At the ICMS Workshop held in November 2024, we presented evidence for two finiteness conjectures around the distribution of heavenly elliptic curves over quadratic number fields, one in terms of isomorphism classes of curves, and a second in terms of quadratic fields which admit heavenly elliptic curves. We prove these two conjectures are equivalent.
\end{abstract}

\begin{CJK*}{UTF8}{goth}

\maketitle

\section{Introduction}\label{sec:intro}

Let us fix an algebraic closure $\Qbar$ of $\QQ$, let $\ell$ be a rational prime, and let $K \subseteq \Qbar$ be a number field. Recall that in the Ihara program to study Galois groups through the natural pro-$\ell$ quotients of fundamental groups, one encounters the tower of fields $K \subseteq \yama \subseteq \ten \subseteq \Qbar$, where $\yama$ (``yama'') and $\ten$ (``ten'') are defined as follows: $\ten := \ten(K, \ell)$ is the maximal pro-$\ell$ extension of $K(\bmu_{\ell^{\infty}}\!)$ unramified away from $\ell$, and $\yama := \yama(K, \ell)$ is the fixed field of the outer pro-$\ell$ Galois representation attached to $\PP^{1}_{K} \smallsetminus \{0, 1, \infty \}$. It is known that $\yama \subseteq \ten$ for any $K$ and $\ell$, and that $\yama(\QQ, \ell) = \ten(\QQ, \ell)$ for odd regular primes $\ell$. (More precisely, Sharifi \cite{Sharifi:2002} proved this is equivalent to the conjecture of Deligne and Ihara, which was later proved by Brown \cite{Brown:2012}.) 

An abelian variety $A/K$ is called \emph{heavenly} at $\ell$ if $K(A[\ell^\infty]) \subseteq \ten$. It follows from Serre-Tate theory that such a variety necessarily has good reduction away from $\ell$. Thus, for fixed $K$, fixed dimension $g$, and fixed prime $\ell$, there can be only finitely many abelian varieties $A/K$ of dimension $g$ which are heavenly at $\ell$.
For fixed $K$ and $g$, let $\mathcal{H}(K,g)$ be the set of pairs $([A]_{K}, \ell)$ where $[A]_{K}$ is the $K$-isomorphism class of a $g$-dimensional abelian variety $A/K$ which is heavenly at $\ell$. In \cite{Rasmussen-Tamagawa:2008}, the second author and Tamagawa conjectured that $\mathcal{H}(K,g)$ is finite for any choice of $K$ and $g$; equivalently, for $\ell \gg_{K,g} 0$, there should be \emph{no} abelian varieties $A/K$ in dimension $g$ which are heavenly at $\ell$. Many partial results towards the conjecture have been obtained, and a description of recent results is given in \cite[\S3]{McLeman-Rasmussen:2024}. 

The article \cite{McLeman-Rasmussen:2024} studies heavenly elliptic curves over quadratic fields. Based on striking connections between the traces of Frobenius elements for curves with complex multiplication and for heavenly elliptic curves, we conjectured the following: excluding infinite families arising with $\ell < 7$, or from base change constructions, there are only finitely many $\Qbar$-isomorphism classes of elliptic curves which are defined over a quadratic field $K$ and which are heavenly at some prime $\ell$. We present the conjecture in two ways, one in terms of the isomorphism classes of heavenly elliptic curves, and the other in terms of the quadratic fields admitting new examples of heavenly elliptic curves. The purpose of the present article is to state both forms precisely and demonstrate their equivalence.


For any number field $K$, integer $g > 0$, and prime $\ell$, let $\mathcal{H}(K, g, \ell)$ denote the set of $K$-isomorphism classes of abelian varieties which are defined over $K$, have dimension $g$, and which are heavenly at $\ell$. For any $B \geq 1$, set
\[ \mathcal{H}_{B} := \bigcup_{\ell \geq B\vphantom{[}} \bigcup_{\substack{K \\ [K:\QQ]=2}} \left\{ \bigl( [A]_{K}, \ell \bigr) : [A]_{K} \in \mathcal{H}(K,1,\ell) \right\}. \]
We also define $\overline{\mathcal{H}}_{B}$, which tracks analogous pairs, but only up to $\Qbar$-isomorphism:
\[ \overline{\mathcal{H}}_{B} := \left\{ \bigl( [A \times_{K} \Qbar]_{\Qbar}, \ell \bigr) : ([A]_{K}, \ell) \in \mathcal{H}_{B} \right\}. \]
Note $\overline{\mathcal{H}}_{B}$ is in obvious bijection with the set of equivalence classes of $\mathcal{H}_{B}$ under the equivalence relation
\[ ([A]_{K}, \ell) \sim ([A']_{K'}, \ell')\ \quad \text{precisely if} \quad \ell = \ell' \ \text{and}\ A \times_K \Qbar \cong_{\Qbar} A' \times_{K'} \Qbar. \]
Note that if $E/\QQ$ is heavenly at some prime $\ell_0$, then the infinite family $\{ ([E \times_{\QQ} K]_{K}, \ell_0) \}_{K}$, indexed by all quadratic fields $K$, is a subset of $\mathcal{H}_{1}$. To exclude the redundancy of such examples, we further define 
\begin{align*}
  \overline{\mathcal{H}}(K, g) & := \{ ([A \times_{K} \Qbar]_{\Qbar}, \ell) : ([A]_{K}, \ell) \in \mathcal{H}(K, g) \}, \\
  \mathcal{H}^{\circ}_{B} & := \left\{ ([A]_{K}, \ell) \in \mathcal{H}_{B} : ([A \times_{K} \Qbar]_{\smash{\Qbar}}, \ell) \not\in \overline{\mathcal{H}}(\QQ, 1) \right\}, \\
  \overline{\mathcal{H}}^{\circ}_{B} & := \left\{ ([A]_{\smash{\Qbar}}, \ell) \in \overline{\mathcal{H}}_{B} : ([A]_{\smash{\Qbar}}, \ell) \not\in \overline{\mathcal{H}}(\QQ, 1) \right\}.
\end{align*}

However, $\overline{\mathcal{H}}^{\circ}_{1}$ is not finite, either. By \cite[Prop.~5.4]{Rasmussen-Tamagawa:2017}, there are infinitely many pairs $([E]_{\Qbar}, \ell_0) \in \overline{\mathcal{H}}_{1}^{\circ}$ with $\ell_0 = 2$; we expect a similar result for $\ell_0 \in \{3, 5\}$ (see \cite[Rmk.~3.5]{McLeman-Rasmussen:2024}). But for each $\ell_0 > 5$, $\overline{\mathcal{H}}_{1}$ contains only finitely many pairs with $\ell = \ell_0$ \cite[Cor.~3.4]{McLeman-Rasmussen:2024}, suggesting the following.

\begin{conjecture}\label{conj:H_finite}
  The set $\overline{\mathcal{H}}_{7}$ is finite.
\end{conjecture}


Let us now describe the distribution of heavenly elliptic curves from a different perspective, by considering the following question: for which choices of a quadratic field $K$ and a rational prime $\ell$ do there exist elliptic curves $E/K$ which are heavenly at $\ell$? For some choices of $K$ and $\ell$, it  holds that every heavenly $E/K$ is $\Qbar$-isomorphic to some heavenly $E_0/\QQ$. Analogous to the restriction placed on $\overline{\mathcal{H}}^{\circ}_{B}$, we refine the question to exclude such pairs. So for any $B \geq 1$, define the set $\mathcal{R}_{B}$ as follows: an ordered pair $(\ell, K) \in \mathcal{R}_{B}$ precisely if
\begin{enumerate}[left=0pt, label={(\roman*)}]
\item $\ell \geq B$ is prime,
\item there exists an elliptic curve $E/K$, heavenly at $\ell$, such that
\item $E \times_{K} \Qbar \not\cong A \times_{\QQ} \Qbar$ for any $A/\QQ$ which is heavenly at $\ell$ over $\QQ$.
\end{enumerate}

\begin{conjecture}\label{conj:R_finite}
  The set $\mathcal{R}_{7}$ is finite.
\end{conjecture}

\begin{remark}
Condition (iii) in the definition of $\mathcal{R}_{B}$ is quite mild. The finite set $\mathcal{H}(\QQ, 1)$ was computed explicitly in \cite{Rasmussen-Tamagawa:2008}, and spans a set $J$ of only $16$ rational $j$-invariants. So (iii) is satisfied automatically if $j(E) \not\in J$.
\end{remark}

\begin{remark}
Theorem \ref{thm:equifinite} below implies that Conjectures \ref{conj:H_finite} and \ref{conj:R_finite} are equivalent. Let us explain briefly why the equivalence is not immediate. The concern is the following: given a single pair $([\overline{E}]_{\Qbar}, \ell) \in \overline{\mathcal{H}}_{7}$, there could exist an infinite sequence $A_{i}/K_{i}$ of elliptic curves defined over quadratic fields, each heavenly at $\ell$ and satisfying $A_{i} \times_{K_{i}} \Qbar \cong \overline{E}$. In this case, $\mathcal{R}_{7}$ possibly contains the infinite family $\{ (\ell, K_{i}) \}$. However, as demonstrated in the proof of Theorem \ref{thm:equifinite}, this is possible only if $\overline{E}$ can also be represented by a heavenly elliptic curve defined over $\QQ$, and consequently the existence of such a pair in $\overline{\mathcal{H}}_{7}$ does \emph{not} automatically imply $\mathcal{R}_{7}$ is infinite.
\end{remark}

\section{Equivalence of Conjectures}

For any prime $\ell > 2$, set $\ell^{\star} := (-1)^{(\ell-1)/2}\ell$, so that $\QQ(\sqrt{\ell^{\star}}) \subseteq \QQ(\bmu_\ell)$. Within $\GL_2(\FF_\ell)$, we let $\mathcal{U}$ and $\mathcal{B}$ denote the subgroup of unit upper triangular matrices, and the Borel subgroup of upper triangular matrices, respectively, so that $\mathcal{U} \leq \mathcal{B} \leq \GL_2(\FF_\ell)$ and $\sharp \mathcal{U} = \ell$. Let $\chi \colon G_{\QQ} \to \FF_{\ell}^\times$ be the $\ell$-adic cyclotomic character modulo $\ell$. We recall the following fact, which is a specific case of \cite[Lemma 3.3]{Rasmussen-Tamagawa:2017}.

\begin{lemma}\label{lemma:RT3.3}
  Let $K$ be a number field and let $A/K$ be an elliptic curve which is heavenly at $\ell$. With respect to some basis for $A[\ell]$, the representation $\rho_{A} \colon G_K \to \GL_2(\FF_\ell)$ has the form
  \[ \rho_A = \begin{pmatrix*} \chi^{i_{1}} & \star \\ 0 & \chi^{i_{2}} \end{pmatrix*}, \qquad 0 \leq i_{1}, i_{2} < \ell - 1. \qedhere \]
\end{lemma}
In particular, with respect to this basis, $\rho_A(G_{\QQ}) \leq \mathcal{B}$ and $\rho_A(G_{\QQ(\bmu_{\ell})}) \leq \mathcal{U}$.

We define maps $\alpha$ and $\beta$ as follows: 
  \begin{align*}
    \alpha & \colon \mathcal{H}^{\circ}_{B} \to \overline{\mathcal{H}}^{\circ}_{B}, & ([A]_{K}, \ell) & \mapsto ([A]_{\Qbar}, \ell), \\
    \beta & \colon \mathcal{H}_{B}^{\circ} \to \mathcal{R}_{B}, & ([A]_{K}, \ell) & \mapsto (\ell, K).
  \end{align*}
  It is evident that both maps are surjective. Further, it is straightforward to see that $\beta$ has finite fibers: Suppose $(\ell, K) \in \mathcal{R}_{B}$. If $([A]_{K}, \ell) \in \beta^{-1}((\ell, K))$, then necessarily $A$ has good reduction away from $\ell$.  By the unpolarized version of the Shafarevich Conjecture \cite{Zarhin:1985}, there are only finitely many such elliptic curves defined over $K$, up to $K$-isomorphism. Thus, the fiber $\beta^{-1}((\ell, K))$ is finite. It is also true that every fiber of $\alpha$ is finite, though the arguments are rather more intricate. Let us first consider the case of an exceptional $j$-invariant.
  
\begin{lemma}\label{lemma:case_special_j}
Suppose $\bigl([E]_{\Qbar}, \ell \bigr) \in \overline{\mathcal{H}}_3^\circ$ and let $X$ be the fiber of $\alpha$ over $([E]_{\Qbar}, \ell)$. Set $j := j(E)$.
  \begin{enumerate}[left=0pt, label={(\alph*)}]
  \item If $j = 0$ and $\ell > 3$, then $X \subseteq \mathcal{H}(\QQ(\sqrt{-3}), 1, \ell) \cup \mathcal{H}(\QQ(\sqrt{-3\ell^\star}), 1, \ell)$.
  \item If $j = 1728$ and $\ell > 2$, then $X \subseteq \mathcal{H}(\QQ(\sqrt{-1}), 1, \ell) \cup \mathcal{H}( \QQ(\sqrt{-\ell^\star}), 1, \ell)$.
  \end{enumerate}
  In either case, the fiber $X$ is finite.
\end{lemma}

\begin{proof}
  Suppose $j = 0$ and $\ell > 3$. As $([E]_{\Qbar}, \ell) \in \overline{\mathcal{H}}_3^{\smash{\circ}}$, $E$ is defined over some quadratic field $K$. Since $j = 0$, the CM field for $E$ is $\QQ(\bmu_6)$, and so by \cite[Lemma 3.15]{Bourdon-et-al:2017}, $K(\bmu_{6}) \subseteq K(E[\ell])$. Let $\rho_E$ be the $G_K$-representation attached to $E[\ell]$, and let $\Gamma := \Gal(K(E[\ell])/K)$. As $E$ is heavenly, $\Gamma \cong \rho_E(G_K) \leq \mathcal{B}$. Let $\Gamma_1 = \Gamma \cap \mathcal{U}$; clearly $K(\bmu_\ell) = K(E[\ell])^{\Gamma_1}$.
Note that 
  \[ [ K( E[\ell] ) : K( \bmu_{\ell} ) ] = [ K( E[\ell] ) : K( \bmu_{\ell}, \bmu_{6}) ] [ K( \bmu_{\ell}, \bmu_{6}) : K( \bmu_{\ell}) ], \]
  while $[K(\bmu_{\ell}, \bmu_{6}) : K(\bmu_{\ell})] \leq [\QQ(\bmu_{6}) : \QQ] = 2$. Since $\sharp \mathcal{U} = \ell$, it follows that $[ K( E[\ell] ) : K( \bmu_{\ell}) ]$ is odd. So $[ K( \bmu_{\ell}, \bmu_{6}) : K( \bmu_{\ell}) ] = 1$, i.e., $K( \bmu_{6} ) \subseteq K( \bmu_{\ell})$. We conclude $\sqrt{-3}, \sqrt{\vphantom{-}\ell^\star}, \sqrt{-3\ell^\star} \in K(\bmu_{\ell})$. Since $\ell \neq 3$, these three elements generate distinct quadratic subfields of $K(\bmu_{\ell})$. On the other hand, $K(\bmu_\ell) = K \cdot \QQ(\bmu_\ell)$ contains at most three quadratic subfields, so it must be that $K = \QQ(\sqrt{D})$ for some $D \in \{ -3, \ell^\star, -3\ell^\star \}$. Finally, $D = \ell^\star$ would imply $K(\bmu_\ell) = \QQ(\bmu_\ell)$, which does not contain $\bmu_6$ (as $\ell > 3$). So $D \in \{ -3, -3\ell^\star\}$, proving (a).

The proof of (b) is completely analogous to the argument for (a), except that we start by observing the CM field for $E$ is $\QQ(\bmu_{4})$, and deduce $D \in \{-1, -\ell^\star\}$ instead.
\end{proof}

The next proposition demonstrates that the fibers of $\alpha$ are finite for general $j$-invariants also. 

\begin{proposition}\label{prop:case_general_j}
  Suppose $\overline{E}/\Qbar$ is an elliptic curve, $\bigl([\overline{E}]_{\Qbar}, \ell \bigr) \in \overline{\mathcal{H}}^{\circ}_{3}$, and $j := j(\overline{E}) \not\in \{0, 1728 \}$. Let $X$ be the fiber of $\alpha$ over $([\overline{E}]_{\Qbar}, \ell)$.
  \begin{enumerate}[label={(\alph*)}, left=0pt]
    \item Suppose $K', K''$ are distinct quadratic fields and $A'/K'$, $A''/K''$ are elliptic curves satisfying $([A']_{K'}, \ell), ([A'']_{K''}, \ell) \in X$. Then $j(\overline{E}) \in \QQ$.
    \item Let $A_0/\QQ$ be any elliptic curve with $j(\overline{E}) = j(A_0)$. For every $([A]_{K}, \ell) \in X$, it holds that $K \subseteq \QQ(A_0[\ell])$.
    \end{enumerate}
In particular, the set $X$ must be finite.
  \end{proposition}

\begin{proof}
  Let us first explain how the finiteness of $X$ follows from (a) and (b). Towards a contradiction, suppose $X$ is infinite. For any particular quadratic field $K$, $X$ may contain only finitely many pairs of the form $([A]_{K}, \ell)$; each such $A/K$ has good reduction away from $\ell$ and there are only finitely many such curves up to $K$-isomorphism by the Shafarevich Conjecture. Thus, there must exist pairs $([A]_{K}, \ell) \in X$ involving infinitely many distinct quadratic fields $K$. By (a), $j(\overline{E}) \in \QQ$, and we may choose $A_0/\QQ$ with the same $j$-invariant. Now (b) implies a contradiction, since $\QQ(A_0[\ell])$ contains only finitely many quadratic fields.

The claim in (a) is immediate. The curves $\overline{E}$, $A'$, $A''$ necessarily have the same $j$-invariant and so $j(\overline{E}) \in K' \cap K'' = \QQ$.

For (b), let $j = j(\overline{E})$ and take $A_0/\QQ$ with $j(A_0) = j$. Suppose $([A]_{K}, \ell) \in X$, and for the sake of contradiction, assume $K \not\subseteq \QQ(A_0[\ell])$. Since $A_0 \times_\QQ K$ and $A$ are both defined over $K$ and have the same $j$-invariant, they are quadratic twists of one another. Let $\lambda \colon G_K \to \{ \pm 1 \}$ be a quadratic character such that $(A_0 \times_{\QQ} K)^\lambda \cong_K A$. We will demonstrate that $\lambda$ may be extended to a quadratic character $\varepsilon$ on $G_\QQ$ satisfying $A_0^{\varepsilon} \times_{\QQ} K \cong_K A$.

Since $A$ is heavenly at $\ell$ over $K$, Lemma \ref{lemma:RT3.3} implies the associated representation $\rho_A$ on $A[\ell]$ has the form
\[ \rho_A \colon G_K \to \GL_2(\FF_\ell), \qquad \rho = \begin{pmatrix*} \chi^{i_1} & \star \\ 0 & \chi^{i_2} \end{pmatrix*}. \]
Let $\rho_{A_0} \colon G_\QQ \to \GL_2(\FF_\ell)$ be the representation attached to $A_0[\ell]$. Because $A$ is the twist of $A_0 \times_\QQ K$  by $\lambda$, we have
\[ \textstyle \eval{\rho_{A_0}}_{G_K} = \rho_A^\lambda = \begin{pmatrix*} \lambda \chi^{i_1} &  \star \\ 0 & \lambda \chi^{i_2} \end{pmatrix*}. \]
Set $G_0 := \Gal(\QQ(A_0[\ell])/\QQ)$. Since $K$ and $\QQ(A_0[\ell])$ are linearly disjoint, $G_0 \cong \Gal(K(A_0[\ell])/K)$ also, and so $\rho_{A_0}(G_\QQ) \cong \rho_{A_0}(G_K)$. On the other hand, $\rho_{A_0}(G_K) = \rho_A^\lambda(G_K)$ is a subgroup of $\mathcal{B}$ (under some choice of basis). Consequently, $\rho_{A_{0}}(G_{\QQ}) \subseteq \mathcal{B}$. There must exist characters $\phi_{i} \colon G_{\QQ} \to \FF_{\ell}^{\times}$ such that $\rho_{A_{0}}$ is upper triangular with $\phi_{1}$, $\phi_{2}$ on the diagonal. We have now expressed $\eval{\rho_{A_{0}}}_{G_{K}}$ in two distinct ways and conclude that, at least on $G_K$, $\lambda \chi^{i_{m}} = \phi_m$. For $m = 1, 2$, set $\varepsilon_{m} := \phi_{m} \chi^{-i_{m}}$. These are characters defined on $G_{\QQ}$ which satisfy $\eval{\varepsilon_{m}}_{G_{K}} = \lambda$. The $\varepsilon_{m}$ and $\lambda$ all factor through $G_{0}$. Necessarily, the natural projection maps yield the following commutative diagram:
\[ \begin{tikzcd}
    G_K \ar[r, ->>] \ar[rd, "\lambda"'] & G_0 \ar[d] & G_{\QQ} \ar[l, ->>] \ar[ld, "\varepsilon_j"] \\
    & \FF_\ell^\times & 
  \end{tikzcd} \]
We claim $\varepsilon_1 = \varepsilon_2$. For any $\sigma \in G_\QQ$, let $\overline{\sigma} \in G_0$ be the image of $\sigma$ under the natural projection $G_{\QQ} \to G_0$, and pick any $\tilde{\sigma} \in G_K$ that also maps to $\overline{\sigma}$ under the projection $G_{K} \to G_{0}$. Now $\varepsilon_1(\sigma) = \lambda(\tilde{\sigma}) = \varepsilon_2(\sigma)$ and so $\varepsilon_1 = \varepsilon_2$, and we relabel this character $\varepsilon$. Going further, we see $\varepsilon$ is a quadratic character, since $\varepsilon(G_\QQ) \subseteq \lambda(G_K) = \{\pm 1 \}$. Thus $\phi_{m} = \varepsilon \chi^{i_{m}}$. Taking $E_0 := A_0^{\varepsilon}$, the twist of $A_{0}$ by $\varepsilon$, we see this implies
\begin{equation}\label{eqn:rho_A0_both_ways}
  \rho_{A_{0}} = \begin{pmatrix*} \phi_{1} & \star \\ 0 & \phi_{2} \end{pmatrix*} = \begin{pmatrix*} \varepsilon \chi^{i_{1}} & \star \\ 0 & \varepsilon \chi^{i_{2}} \end{pmatrix*}, \qquad \rho_{E_{0}} = \begin{pmatrix*} \chi^{i_{1}} & \star \\ 0 & \chi^{i_{2}} \end{pmatrix*}.
\end{equation}
 But $\eval{\varepsilon}_{G_{K}} = \lambda$, so $(A_{0} \times_{\QQ} K)^{\lambda} \cong_{K} (A_{0}^{\varepsilon} \times_{\QQ} K) = E_{0} \times_{\QQ} K$. From this it follows that $E_{0} \times_{\QQ} \Qbar \cong \overline{E}$, and also that $K(A[\ell^{\infty}]) = K \cdot \QQ(E_{0}[\ell^{\infty}])$.

On $G_{\QQ(\bmu_\ell)}$, $\rho_{E_0}$ is unit upper-triangular, so $[\QQ(E_0[\ell]):\QQ(\bmu_\ell)]$ divides $\ell$. To show $E_0$ is heavenly at $\ell$, it suffices to show $E_0$ has good reduction away from $\ell$ \cite[Lemma 2.5]{McLeman-Rasmussen:2024}. Let $p \neq \ell$ be a rational prime. Since $A/K$ is heavenly at $\ell$, $A$ has good reduction at primes of $K$ above $p$. Take $n > 1$ and set $F = \QQ(E_0[\ell^{n}])$. Let $\pp$ be a prime of $F$ above $p$. As $KF \subseteq K(A[\ell^{\infty}])$, the extension $KF/K$ is unramified at primes of $K$ above $p$, and consequently the ramification index $e_{\pp\,|\,p} \leq 2$. Since $\QQ(\bmu_{\ell})/\QQ$ is unramified at $p$, we conclude $e_{\pp\,|\,p}$ divides $[F:\QQ(\bmu_{\ell})]$. However, $[ F : \QQ( E_{0}[\ell] ) ]$ is a power of $\ell$, hence odd, and $[ \QQ(E_{0}[\ell]) : \QQ(\bmu_{\ell})]$ divides $\ell$. Thus $e_{\pp\,|\,p} = 1$ and, as $n$ was arbitrary, $\QQ(E_0[\ell^{\infty}])/\QQ$ is unramified at $p$. By the Criterion of N\'{e}ron-Ogg-Shafarevich  \cite[Thm.~1]{Serre-Tate:1968}, $E_{0}$ has good reduction at $p$. Thus $E_0$ is heavenly at $\ell$ over $\QQ$ and $([E_0], \ell) \in \mathcal{H}(\QQ, 1)$. However, this implies $([\overline{E}]_{\Qbar}, \ell) \in \overline{\mathcal{H}}(\QQ, 1)$, contradicting the assumption that $([\overline{E}]_{\Qbar}, \ell) \in \overline{\mathcal{H}}^{\circ}_{3}$.
\end{proof}

The hard work is over, and we may verify that the two conjectures above are, in fact, equivalent.

\begin{theorem}\label{thm:equifinite}
Assume $B > 3$. If any one of the four sets $\overline{\mathcal{H}}_{B}$, $\overline{\mathcal{H}}^{\circ}_{B}$, $\mathcal{H}^{\circ}_{B}$, and $\mathcal{R}_{B}$ is finite, then all four sets are finite. In particular, Conjectures \ref{conj:H_finite} and \ref{conj:R_finite} are equivalent.
\end{theorem}

\begin{proof}
  Because $\beta$ is surjective with finite fibers, it holds that $\mathcal{R}_{B}$ is finite if and only if $\mathcal{H}^{\circ}_{B}$ is finite. Likewise, $\alpha$ is surjective with finite fibers and so $\mathcal{H}^{\circ}_{B}$ is finite if and only if $\overline{\mathcal{H}}^{\circ}_{B}$ is finite. Finally, we note that $\overline{\mathcal{H}}^{\circ}_{B} \subseteq \overline{\mathcal{H}}_{B} \subseteq \overline{\mathcal{H}}^{\circ}_{B} \cup \overline{\mathcal{H}}(\QQ, 1)$. Since $\mathcal{H}(\QQ, 1)$ is finite, the set $\overline{\mathcal{H}}(\QQ, 1)$ must be finite. It follows that $\overline{\mathcal{H}}^{\circ}_{B}$ is finite if and only if $\overline{\mathcal{H}}_{B}$ is finite.
\end{proof}

\section*{Acknowledgements}

The authors thank Akio Tamagawa for many helpful discussions, and express their appreciation to an anonymous referee who provided many helpful suggestions that improved the exposition and precision in the paper.

\printbibliography

\end{CJK*}
\end{document}
